\documentclass[11pt,twoside, leqno]{article}

\usepackage{amssymb}
\usepackage{amsmath}
\usepackage{amsthm}
\usepackage{color}
\usepackage{graphicx}

\allowdisplaybreaks

\def\rr{{\mathbb R}}

\def\zz{{\mathbb Z}}
\def\nn{{\mathbb N}}

\def\cc{{\mathbb C}}

\def\cx{{\mathcal X}}

\def\cf{{\mathcal F}}

\def\cm{{\mathcal M}}

\def\fz{\infty}

\def\tz{\Theta}

\def\ls{\lesssim}

\def\rr{{\mathbb R}}
\def\cc{{\mathbb C}}
\def\zz{{\mathbb Z}}
\def\nn{{\mathbb N}}
\def\cm{{\mathcal M}}

\def\fz{\infty}

\def\tz{\theta}

\def\ls{\lesssim}

\def\laz{\langle}
\def\raz{\rangle}

\def\r{\right}
\def\lf{\left}

\def\ra{\rangle}

\newtheorem{theorem}{Theorem}[section]

\newtheorem{corollary}[theorem]{Corollary}
\newtheorem{proposition}[theorem]{Proposition}

\theoremstyle{definition}
\newtheorem{remark}[theorem]{Remark}
\newtheorem{definition}[theorem]{Definition}
\numberwithin{equation}{section}

\begin{document}

\title{\bf\Large A note on complex interpolation and Calder\'on
product of quasi-Banach spaces
\footnotetext{\hspace{-0.35cm} 2010 {\it
Mathematics Subject Classification}. Primary 46B70.
\endgraf
{\it Key words and phrases}.
complex interpolation, Calder\'on product, quasi-Banach lattice
\endgraf
Wen Yuan is supported by the National
Natural Science Foundation  of China (Grant No. 11101038), the Specialized Research Fund for the Doctoral Program of Higher Education
of China (Grant No. 20120003110003), the Fundamental Research Funds for
Central Universities of China (Grant No. 2012LYB26) and the
Alexander von Humboldt Foundation.
}}
\date{}
\author{Wen Yuan}
\maketitle

\begin{center}
\begin{minipage}{13.5cm}{\small
{\noindent{\bf Abstract}\quad
In this paper, we prove that the inner complex interpolation of two
quasi-Banach lattices coincides with the closure of their intersection
in their Calder\'on product. This generalizes a classical result by Shestakov in 1974
for Banach lattices.
}}
\end{minipage}
\end{center}

\arraycolsep=1pt


\section{Introduction}

In this paper we consider the relation between complex interpolations and
Calder\'on products for quasi-Banach lattices. We begin with the definition of complex interpolation for quasi-Banach spaces (see, for example, \cite{ca64,km98,kmm}).
Consider a \emph{couple of quasi-Banach spaces} $X_0,X_1$,
which are continuously embedding into a large topological
vector space $Y$. The \emph{space} $X_0+X_1$ is defined as
$$X_0+X_1:=\{h\in Y:\ \exists\ h_i\in X_i,\ i\in\{0,1\},\ {\rm such\ that}\ h=h_0+h_1\},$$
with
$$\|h\|_{X_0+X_1}:=\inf\{\|h_0\|_{X_0}+\|h_1\|_{X_1}:\ h=h_0+h_1,
\ h_0\in X_1\ {\rm and}\ h_1\in X_1\}.$$

Let $U:= \{z \in \cc : \: 0<\Re e\, z<1\}$ and
$\overline{U}:=\{z\in\cc :\: 0\le \Re e\, z\le 1\}.$
A map
$f$: $U\to X$ is said to be \emph{analytic} if, for any given $z_0\in U$, there exists
$\eta\in(0,\fz)$ such that
$f(z)=\sum_{j=0}^\fz h_n(z-z_0)^n,\ h_n\in X$, is uniformly convergent for
$|z-z_0|<\eta$. A quasi-Banach space $X$ is said to be \emph{analytically convex} if
there exists a positive constant $C$ such that, for any analytic function
$f:\ U\to X$ which is continuous on the closed strip $\overline{U}$,
$$\max_{z\in U}\|f(z)\|_X \le C\max_{{\Re e}\,z\in\{0,1\}}\|f(z)\|_X.$$

Suppose that $X_0+X_1$ is analytically convex.
The \emph{set} $\cf:=\cf(X_0,X_1)$ is defined to be the set
of all functions $f$:\ $U\to X_0+X_1$ satisfying that
\begin{enumerate}
\item[(i)] $f$ is analytic and \emph{bounded} in $X_0+X_1$, which means
that $f(U):=\{f(z):\ z\in U\}$
is a bounded set of $X_0+X_1$;
\item[(ii)]
$f$ is extended continuously to the closure $\overline{U}$ of the strip $U$
such that the traces
$t\mapsto f(j+it)$ are bounded continuous functions into $X_j$,
$j\in\{0,1\}$, $t\in\rr$.
\end{enumerate}
We endow $\cf$ with the \emph{quasi-norm}
$$\|f\|_\cf:=\max\lf\{\sup_{t\in\rr}\|f(it)\|_{X_0},
\ \ \sup_{t\in\rr}\|f(1+it)\|_{X_1}\r\}.$$
Let $\cf_0:=\cf_0(X_0,X_1)$ be closure of all functions $f\in \cf$ such that
$f(z)\in X_0\cap X_1$ for all $z\in U$.
We now recall the definition of complex interpolations.

 \begin{definition}
Let $X_0,\,X_1$ be two quasi-Banach spaces such that
$X_0+X_1$ is analytically convex.
Then the \emph{outer complex interpolation space}
$[X_0,X_1]_\tz$ with $\tz\in(0,1)$ is defined by
$$[X_0,X_1]_\tz:=\{g\in X_0+X_1:\ \exists\ f\in\cf\ {\rm such\ that}\ f(\tz)=g\}$$
and its \emph{norm} given by
$\|g\|_{[X_0,X_1]_\tz}:=\inf_{f\in\cf}\{\|f\|_\cf:f(\tz)=g\}.$
The \emph{inner complex interpolation space}
$[X_0,X_1]^i_\tz$ with $\tz\in(0,1)$ is defined via the same as $[X_0,X_1]_\tz$ with
$\cf$ replaced by $\cf_0$.
\end{definition}

It easily follows from the definition that $[X_0,X_1]^i_\tz
\hookrightarrow[X_0,X_1]_\tz$ and
$X_0\cap X_1$ is dense in $[X_0,X_1]^i_\tz$.
If $X_0$ and $X_1$ are both Banach spaces, then it is known that
the inner and outer complex methods coincide (see \cite{ca64,kmm}).
For the general quasi-Banach cases, Kalton, Mayboroda and Mitrea \cite{kmm} pointed out that the inner and the outer complex methods yield the same space if $X_0$
and $X_1$ are separable analytically convex quasi-Banach spaces.
However,  for quasi-Banach spaces without the separability condition,
whether these  two methods
still coincide  is still unclear (see \cite{kmm}).

Let $(\Omega,\mu)$ be a $\sigma$-finite measure space and $L_0$ be the collection of
all complex-valued $\mu$-measurable functions on $\Omega$. A quasi-Banach function
space $X$ on $\Omega$ is called a {\it quasi-Banach lattice}
if for every $f\in X$ and $g\in L_0$ with $|g(x)|\le |f(x)|$ for
$\mu$-a.e. $x\in \Omega$,
one has $g\in X$ and $\|g\|_X\le \|f\|_X.$

\begin{definition}
Let $X_j \subset L_0$, $j\in\{0,1\}$,  be quasi-Banach lattices on $(\Omega,\mu)$
and $\theta\in(0,1)$. Then the {\it Calder\'on product $X_0^{1-\theta}X_1^\theta$}
of $X_0$ and $X_1$ is the collection of all functions $f \in L_0$ such that
\begin{eqnarray*}
\|f\|_{X_0^{1-\theta}X_1^\theta} := \inf\Bigl\{\|f_0\|_{X_0}^{1-\theta}\|f_1\|_{X_1}^\theta:\:
|f|\le |f_0|^{1-\theta}|f_1|^\theta \quad \mu \mbox{-a.e.},\ \
 f_j\in X_j, \, j\in\{0,1\}\Bigr\}
\end{eqnarray*}
is finite.
\end{definition}

The first result concerning the relation between complex interpolations and
Calder\'on products is due to Calder\'on \cite{ca64}. He showed that if $X_0$ and $X_1$
are Banach lattices, then  $[X_0,X_1]_\tz \hookrightarrow X_0^{1-\tz}X_1^\tz$.
Later, Shestakov \cite{she74} (see also \cite{s74,n85}) in 1974
proved that the complex interpolation of two Banach lattices $X_0$ and $X_1$
is just the closure of their intersection $X_0\cap X_1$ in their Calder\'on product, namely,
$[X_0,X_1]_\tz =\overline{X_0\cap X_1}^{\|\cdot\|_{X_0^{1-\tz}X_1^\tz}}$.
In 1998, Kalton and Mitrea \cite{km98}  considered more general
quasi-Banach cases.  Indeed, they proved in \cite[Theorem 3.4]{km98} that,
\emph{if $X_0$ and $X_1$ are analytically convex separable quasi-Banach lattices, then
$X_0+X_1$ is also analytically convex and
$[X_0,X_1]_\tz=X_0^{1-\tz}X_1^\tz$}. The proof of this result was noticed later in
\cite{kmm} to be also feasible for the coincidence
$[X_0,X_1]^i_\tz=X_0^{1-\tz}X_1^\tz$, and so in this case,
$$[X_0,X_1]_\tz =[X_0,X_1]^i_\tz =X_0^{1-\tz}X_1^\tz=\overline{X_0\cap X_1}^{\|\cdot\|_{X_0^{1-\tz}X_1^\tz}}.$$

Notice that in Kalton and Mitrea's result \cite[Theorem 3.4]{km98}, there is a condition
on the separability of the spaces $X_0$ and $X_1$. An interesting
question is, how is the relation between complex interpolations and Calder\'on
products of quasi-Banach lattices which are not separable?
Is Shestakov's result for Banach spaces also true for general quasi-Banach cases?
In this note we give a positive answer  for the inner complex interpolation.

\begin{theorem}\label{main}
Let $\Omega$ be a Polish space, $\mu$ a $\sigma$-finite Borel measure
on $\Omega$ and $(X_0, X_1)$ a pair of quasi-Banach lattices on $(\Omega,\mu)$.
If both $X_0$ and $X_1$  are analytically convex, then
$$[X_0,X_1]_\theta^i =\overline{X_0\cap X_1}^{\|\cdot\|_{X_0^{1-\tz}X_1^\tz}},\quad \tz\in(0,1).$$
\end{theorem}

To prove Theorem \ref{main}, we use another interpolation method, the Gagliardo-Peetre interpolation introduced by Peetre \cite{p71}.  In 1985 Nilsson \cite{n85} proved
a general result concerning the
relation between Gagliardo-Peetre interpolation and Calder\'on product,
which is a key tool used in this paper. This proof is different from the one
used by Shestakov \cite{she74}
for Banach lattices.

Throughout the paper,
the \emph{symbol}  $C $ denotes   a positive constant
which  may vary from line to line.
The \emph{meaning of $A \ls B$} is
given by: there exists a positive constant $C$ such that
 $A \le C \,B$.
The \emph{symbol $A \sim B$} means
$A \ls B \ls A$.

\section{Proof of Theorem \ref{main}}

Let $X$ be a quasi-Banach lattice and $p\in[1,\fz]$.
The \emph{$p$-convexification} of $X$, denoted by $X^{(p)}$,
is defined as follows: $f\in X^{(p)}$ if and only if
$|f|^p\in X$. For all $f\in X^{(p)}$, define $\|f\|_{X^{(p)}}:=\||f|^p\|_X^{1/p}$.
The lattice $X$ is called \emph{$1/p$-convex} if $X^{(p)}$ is a Banach space.
Moreover, a quasi-Banach  lattice $X$ is said to \emph{be of type $\mathfrak{E}$},
if there exists an equivalent quasi-norm $|||\cdot|||_X$ such that $(X,|||\cdot|||_X)$
is {$1/p$-convex} for some $p\in[1,\fz)$; see \cite{n85}.

An important tool we used is the following equivalent characterization of
analytically convex quasi-Banach lattice; see, for example, \cite{k86,kmm}.
In what follows, $K_X$ denotes the \emph{modulus of concavity} of a
quasi-Banach space $X$, i.\,e., the smallest positive constant satisfying
$$\|x+y\|_X\le K_X(\|x\|_X+\|y\|_X),\quad x,y\in X.$$

\begin{proposition}\label{r-convex}
Let $X$ be a quasi-Banach lattice. Then the following assertions  are equivalent:

{\rm(i)} $X$ is analytically convex;

{\rm(ii)} there exists $r>0$ such that $X$ is $r$-convex, namely, $X^{(1/r)}$ is a Banach space;

{\rm(iii)} $X$ is $r$-convex for all $0<r<(1+\log_2K_X)^{-1}$.
\end{proposition}

It follows from Proposition \ref{r-convex} that all analytically convex
quasi-Banach lattices are of type $\mathfrak{E}$.

We now prove one direction of Theorem \ref{main} in the following theorem, which can be proved by an argument similar to  that used for \cite[Theorem 3.4]{km98}.
For the sake of convenience, we give some details here.

\begin{theorem}
Let $\Omega$ be a Polish space, $\mu$ a $\sigma$-finite Borel measure
on $\Omega$ and $(X_0, X_1)$ a pair of quasi-Banach lattices of functions on $(\Omega,\mu)$.
If both $X_0$ and $X_1$  are analytically convex, then
$$[X_0,X_1]_\theta^i \hookrightarrow \overline{X_0\cap X_1}^{\|\cdot\|_{X_0^{1-\tz}X_1^\tz}},\quad \tz\in(0,1).$$
\end{theorem}

\begin{proof}
Since the lattices  $X_0$, $X_1$ are analytically convex, by Proposition
\ref{r-convex}, we know
  that there exists $r\in(0,1]$ such that $X_0$, $X_1$ are both $r$-convex lattices.
By \cite[Theorem 3.4]{km98} and its proof, $X_0+X_1$ is also $r$-convex and hence $(X_0+X_1)^{(1/r)}$ is a Banach space.

Since $X_0\cap X_1$ is dense in $[X_0,X_1]_\theta^i$, it suffices to prove
$$(X_0\cap X_1, \|\cdot\|_{[X_0,X_1]_\theta^i})\hookrightarrow \overline{X_0\cap X_1}^{\|\cdot\|_{X_0^{1-\tz}X_1^\tz}}.$$
Let $f\in X_0\cap X_1$. Then for any $\varepsilon>0$
there exists $F\in \cf_0(X_0,X_1)$ such that $F(\tz)=f$ and $\|F\|_{\cf(X_0,X_1)}\le \|f\|_{[X_0,X_1]_\theta^i}+\varepsilon$.
Since $F$ is analytic in $U$ and  continuous in  $\overline{U}$,
for any $z_0\in U$, there exist $R>0$
and $f_k\in X_0+X_1$  such that
$F(z)=\sum_{k\in\nn_0} f_k (z-z_0)^k$
with uniformly convergent in $X_0+X_1$ for all $|z-z_0|<R$. Moreover, due to the Cauchy-Hadamard theorem, it holds  $\limsup_{k\to\fz}\|f_k\|_{X_0+X_1}^{1/k}\le R^{-1}$.
We also know that  for $\mu$-almost every $w\in \Omega$,  any $\rho<R$ and $q\le r$,
$$|F(z_0)(w)|^q\le \frac1{2\pi}\int_0^{2\pi} |F(z_0+\rho e^{it})(w)|^q\,dt.$$
Since $z\mapsto |F(z)|^q$ is continuous into $(X_0+X_1)^{(1/q)}$,
we know that, for any positive continuous functional $\phi\in ((X_0+X_1)^{(1/q)})^*$,
\begin{eqnarray*}
\phi(|F(z_0)|^q)&&\le \frac1{2\pi}\phi\left(\int_0^{2\pi} |F(z_0+\rho e^{it})|^q\,dt\right)\le\frac1{2\pi}\int_0^{2\pi} \phi\left(|F(z_0+\rho e^{it})|^q\right)\,dt,
\end{eqnarray*}
hence $z \mapsto \phi(|F(z)|^q)$ is subharmonic on $U$.  Then
\begin{eqnarray}\label{e1}
\phi(|F(\tz)|^q)&&\le \int_\rr P_0(\tz,t)\phi(|F(it)|^q)\,dt+\int_\rr P_1(\tz,t)\phi(|F(1+it)|^q)\,dt,
\end{eqnarray}
where $P_0$ and $P_1$ are the components of the
Poisson kernel on $U$ satisfying
$\int_\rr P_0(\tz,t)\,dt=1-\tz$ and $\int_\rr P_1(\tz,t)\,dt=\tz$.
Let $f_0:=((1-\tz)^{-1}\int_\rr P_0(\tz,t)|F(it)|^r\,dt)^{1/r}$ and $f_1:=((1-\tz)^{-1}\int_\rr P_0(\tz,t)|F(it)|^r\,dt)^{1/r}$. It follows from the
$r$-convexity of $X_0$ and $X_1$ that $f_j\in X_j$ with $\|f_j\|_{X_j}\le \|F\|_{\cf(X_0,X_1)}$, $j\in\{0,1\}$. By \eqref{e1}, $q\le r$
 and the positivity of $\phi$, we have $|F(\tz)|^q\le (1-\tz)f_0^q+\tz f_1^q$. Taking $log$ in both side and letting $q\to0$  then gives  $|f|=|F(\tz)|\le f_0^{1-\tz}f_1^\tz$.
Thus, $\|f\|_{X_0^{1-\tz}X_1^\tz}\le \|F\|_{\cf(X_0,X_1)}\ls \|f\|_{[X_0,X_1]_\tz^i}$, as desired.
\end{proof}

To show the other direction, we need the following Gagliardo-Peetre interpolation method,
which was introduced by Peetre \cite{p71}.

\begin{definition}
Let $X_0$ and $X_1$ be a pair quasi-Banach spaces  and $\theta \in (0,1)$.
We say $a\in \laz X_0, X_1\raz_\theta$ if there exists a sequence $\{a_i\}_{i\in\zz}
\subset X_0\cap X_1$ such that $a=\sum_{i\in\zz} \, a_i$ with convergence in $X_0+X_1$
and for any bounded sequence $\{\varepsilon_i\}_{i\in\zz}\subset\cc$,
$\sum_{i\in\zz} \varepsilon_i \, 2^{i(j-\theta)} \, a_i$
converges in $X_j$, $j\in\{0,1\}$. We further require that
\[
\lf\|\sum_{i\in\zz} \varepsilon_i \, 2^{i(j-\theta)} \, a_i\r\|_{X_j}\le
C \, \sup_{i\in\zz}|\varepsilon_i|,\quad j\in\{0,1\},
\]
for some constant $C$. As a  quasi-norm of
$\laz X_0, X_1\raz_\tz$, we use $\|a\|_{\laz X_0, X_1\raz_\Theta}:=\inf C$
\end{definition}

Applying Proposition \ref{r-convex} and \cite[Theorem 2.1]{n85} (see \cite[(2.1)]{n85}),
we have the following  conclusion.

\begin{theorem}
Let $\Omega$ be a Polish space, $\mu$ a $\sigma$-finite Borel measure
on $\Omega$ and $(X_0, X_1)$ a pair of quasi-Banach lattices of functions on $(\Omega,\mu)$.
If both $X_0$ and $X_1$  are analytically convex, then
$$\laz X_0,X_1\raz_\tz =\overline{X_0\cap X_1}^{\|\cdot\|_{X_0^{1-\tz}X_1^\tz}},\quad \tz\in(0,1),$$
and $\|\cdot\|_{\laz X_0,X_1\raz_\tz}$ is equivalent to $\|\cdot\|_{X_0^{1-\tz}X_1^\tz}$.
\end{theorem}

From this conclusion, we
deduce that $X_0\cap X_1$ is dense in $\laz X_0,X_1\raz_\tz$
and, to prove  $\overline{X_0\cap X_1}^{\|\cdot\|_{X_0^{1-\tz}X_1^\tz}}\hookrightarrow[X_0,X_1]_\theta^i$, it
suffices to show
$\laz X_0,X_1\raz_\tz\hookrightarrow[X_0,X_1]_\theta^i$.

\begin{theorem}
Let $\Omega$ be a Polish space, $\mu$ a $\sigma$-finite Borel measure
on $\Omega$ and $(X_0, X_1)$ a pair of quasi-Banach lattices of functions on $(\Omega,\mu)$.
If both $X_0$ and $X_1$  are analytically convex, then
$$\laz X_0,X_1\raz_\tz\hookrightarrow[X_0,X_1]_\theta^i,\quad \tz\in(0,1).$$
\end{theorem}

\begin{proof} Let $D(X_0,X_1,\tz)$ be the subspace of $\laz X_0,X_1\ra_\tz$
consisting of all $f\in \laz X_0,X_1\raz_\tz$ such that there exists
a finite set $E\subset \zz$ and $\{f_k\}_{k\in E}\subset X_0\cap X_1$ such that
$f=\sum_{k\in E} f_k$ in $X_0+X_1$, and for any bounded sequence
$\{\varepsilon_k\}_{k\in E}$ of complex numbers
$\sum_{k\in E} \varepsilon_k 2^{k(j-\tz)}f_k$ converges in $X_j$, with
$$\lf\|\sum_{k\in E} \varepsilon_k 2^{k(j-\tz)}f_k\r\|_{X_j}\ls \|f\|_{\laz X_0,X_1\raz_\tz} \sup_{k\in E}|\varepsilon_k|, \quad j\in\{0,1\}.$$
Obviously, $X_0\cap X_1 \subset D(X_0,X_1,\tz)$, and hence $D(X_0,X_1, \tz)$
is  dense  in $\laz X_0,X_1\raz_\tz$. To complete the proof,
it suffices to show
$$(D(X_0,X_1,\tz),\|\cdot\|_{\laz X_0,X_1\raz_\tz})\hookrightarrow[X_0,X_1]_\theta^i.$$

Let $f\in D(X_0,X_1,\tz).$ Without loss of generality, we may assume that
$f=\sum_{|k|\le M} f_k$ in $X_0+X_1$ for some $M\in\nn$ and $\{f_k\}_{|k|\le M}\subset X_0\cap X_1$, and
\begin{eqnarray}\label{con1}
\lf\|\sum_{|k|\le M} \varepsilon_k 2^{k(j-\tz)}f_k\r\|_{X_j}\ls \|f\|_{\laz X_0,X_1\raz_{\tz}} \sup_{k\in E}|\varepsilon_k|, \quad j\in\{0,1\}.
\end{eqnarray}
Define $F(z):=\sum_{|k|\le M} 2^{k(z-\tz)}f_k$ with convergence in $X_0+X_1$
for all $z\in \overline{U}$. Obviously, $F(\tz)=f$ and $F(z)\in X_0\cap X_1$.

Now we prove $F\in \cf_0(X_0,X_1)$. The analyticity of $F$ is obvious.
To show $F$ is bounded in $X_0+X_1$, for $z\in \overline{U}$,
write $z=a+ib$ with $a\in[0,1]$ and $b\in\rr$, and
\begin{eqnarray*}
F(z)=\sum_{-M\le k<0} 2^{ka+kbi} 2^{-k\tz}f_k +\sum_{0\le k\le M}2^{k(a-1)+kbi}
2^{k(1-\tz)}f_k=:F_0(z)+F_1(z).
\end{eqnarray*}
Since $\{2^{ka+kbi}\}_{-M\le k<0}$ and $\{2^{k(a-1)+kbi}\}_{0\le k\le M}$ are bounded sequences, by \eqref{con1}, we have
$$\|F_j(z)\|_{X_j}\ls \|f\|_{\laz X_0,X_1\raz_\tz}, \quad j\in\{0,1\}.$$
This implies $F(z)\in X_0+X_1$ and $\|F(z)\|_{X_0+X_1}\ls \|f\|_{\laz X_0,X_1\raz_\tz}$ for all $z\in \overline{U}$.

Similarly, since $\{2^{kti}\}_{|k|\le M}$ is a bounded sequence, applying
\eqref{con1} we obtain
\begin{eqnarray*}
\|F(j+it)\|_{X_j}&&=\left\|\sum_{|k|\le M} 2^{kit}2^{k(j-\tz)}f_k\r\|_{X_j} \\
&&\ls\|f\|_{\laz X_0,X_1\raz_\tz} \sup_{|k|\le M}|2^{kti}|\ls
\|f\|_{\laz X_0,X_1\raz_\tz}, \quad j\in\{0,1\}.
\end{eqnarray*}

Now we show $t\mapsto F(j+it)$  is a
continuous function into $X_j$, $j\in\{0,1\}$.
Fix $t_0\in\rr$. Notice that, for any $\varepsilon>0$, we can find
$\delta=\delta(M,\varepsilon)>0$, such that
for any $|t-t_0|<\delta$ and $|k|\le M$,
$|2^{kit}-2^{kit_0}|<\varepsilon$. Hence,
\begin{eqnarray*}
\|F(j+it)-F(j+it_0)\|_{X_j}&&=\left\|\sum_{|k|\le M} [2^{kit}-2^{kit_0}]2^{k(j-\tz)}f_k\r\|_{X_j} \\
&&\ls\|f\|_{\laz X_0,X_1\raz_\tz} \sup_{|k|\le M}|2^{kti}-2^{kit_0}|\ls
\varepsilon\|f\|_{\laz X_0,X_1\raz_\tz}, \quad j\in\{0,1\},
\end{eqnarray*}
as desired.

It remains to show the extension of $F$ from $U$ to  $\overline{U}$ is continuous. Since $F$ is analytic in $U$, we only need to prove that, for any $t\in\rr$,
\begin{equation}\label{con2}
\|F(a+it)-F(it)\|_{X_0+X_1}\to 0, \quad a\to 0^+
\end{equation}
and
\begin{equation}\label{con3}
\|F(a+it)-F(1+it)\|_{X_0+X_1}\to 0, \quad a\to 1^-.
\end{equation}
For any $\varepsilon>0$, we can find
$\delta=\delta(M,\varepsilon)>0$, such that
for any $0<a<\delta$ and $|k|\le M$,
$|2^{ka}-1|<\varepsilon$.
Write
\begin{eqnarray*}
F(a+it)-F(it)
&&=\sum_{-M\le k<0}[2^{ka}-1]2^{kit}2^{-k\tz}f_k+\sum_{|k|\le M} [2^{ka}-1]2^{-k}2^{kit}2^{k(1-\tz)}f_k.
\end{eqnarray*}
Since
$$\lf\|\sum_{-M\le k<0}[2^{ka}-1]2^{kit}2^{-k\tz}f_k\r\|_{X_0}\ls \varepsilon \|f\|_{\laz X_0,X_1\raz_\tz}$$
and
$$\lf\|\sum_{0\le k\le M}[2^{ka}-1]2^{-k}
2^{kit}2^{k(1-\tz)}f_k\r\|_{X_1}\ls \varepsilon \|f\|_{\laz X_0,X_1\raz_\tz},$$
we know that
$$\|F(a+it)-F(it)\|_{X_0+X_1}\ls  \varepsilon \|f\|_{\laz X_0,X_1\raz_\tz}.$$
This gives \eqref{con2}. A similar argument gives \eqref{con3}.

Combining the above arguments, we know that $F\in \cf_0(X_0,X_1)$ with
$\|F\|_{\cf(X_0,X_1)}\ls \|f\|_{\laz X_0,X_1\raz_\tz}$.
Therefore, $\|f\|_{[X_0,X_1]_\tz^i}\ls \|f\|_{\laz X_0,X_1\raz_\tz}$.
This finishes the proof.
\end{proof}

Theorem \ref{main} is then a consequence of the above three theorems.
Moreover, as a byproduct, we obtain the coincidence between the inner complex interpolation
and the Gagliardo-Peetre interpolation.

\begin{corollary}
Let $(X_0, X_1)$ a pair of analytically convex
quasi-Banach lattices of functions on $(\Omega,\mu)$.
Then $\laz X_0,X_1\raz_\tz=[X_0,X_1]_\tz^i$ for all $\tz\in(0,1).$
\end{corollary}

\begin{remark}
Recall that if $X_0$ and $X_1$ are Banach spaces, then
$\laz X_0,X_1\raz_\tz\hookrightarrow [X_0,X_1]_\tz$; see, for example, \cite{p71,j81,n85}.
The above corollary gives a generalization of this coincidence.
\end{remark}

From the relation between the inner and outer complex interpolations (see \cite{ca64} and
\cite[Theorem 7.9]{kmm}), we also have the following conclusion.

\begin{corollary}
Let $\tz\in(0,1)$ and $(X_0, X_1)$ a pair of analytically convex
quasi-Banach lattices of functions on $(\Omega,\mu)$.
Then $\laz X_0,X_1\raz_\tz=[X_0,X_1]_\tz$ if either $X_0$, $X_1$ are both
Banach spaces or $X_0,X_1$ are both separable.
\end{corollary}

Finally we give an application of Theorem \ref{main} to
the Morrey space, which is a  typical example of non-separable spaces.
Let  $0<p\le u\le\infty$ and $(\cx,\mu)$
be a quasi-metric measure space. Recall that
the Morrey space $\mathcal{M}^u_{p}(\cx)$
is  the collection of all $p$-locally integrable
functions $f$ on $\cx$ such that
\begin{equation*}
\|f\|_{\mathcal{M}^u_{p}(\cx)} :=  \sup_{B\subset \cx}
|B|^{1/u-1/p}\lf[\int_B |f(x)|^p\,dx\r]^{1/p}<\infty\, ,
\end{equation*}
where the supremum is taken over all balls $B$ in $\cx$.
Obviously, $\mathcal{M}^p_{p}(\cx)=L_p(\cx)$.
Since Morrey spaces are non-separable, we can not apply \cite[Theorem 3.4]{km98} to Morrey spaces.

By \cite[Proposition 2.1]{lyy}, we know that $[\cm_{p_0}^{u_0}(\cx)]^{1-\tz}[\cm_{p_1}^{u_1}(\cx)]^\tz=\cm_{p}^{u}(\cx)$, which together with Theorem \ref{main} induce the following conclusion.

\begin{proposition}
Let $\tz\in(0,1)$, $0 < p_i \le u_i <\infty$, $i\in\{0,1\}$ and
$$\frac 1u := \frac{1-\tz}{u_0} + \frac{\tz}{u_1} \, ,
\quad \frac 1p := \frac{1-\tz}{p_0} + \frac{\tz}{p_1}.$$
If $u_0p_1=u_1p_0$, then
$$[\cm_{p_0}^{u_0}(\cx),\cm_{p_1}^{u_1}(\cx)]_\theta^i=\laz \cm_{p_0}^{u_0}(\cx),\cm_{p_1}^{u_1}(\cx)\raz_\tz =\overline{\cm_{p_0}^{u_0}(\cx)\cap
\cm_{p_1}^{u_1}(\cx)}^{\|\cdot\|_{\cm_{p}^{u}(\cx)}}.$$
\end{proposition}

\bigskip

\noindent  Wen Yuan

\medskip

\noindent  School of Mathematical Sciences, Beijing Normal University,
Laboratory of Mathematics and Complex Systems, Ministry of
Education, Beijing 100875, People's Republic of China

\noindent {\it E-mail}: \texttt{wenyuan@bnu.edu.cn}
\end{document}